\documentclass [12pt,oneside]{amsart}
\usepackage{xcolor}
\usepackage{amssymb,amsmath,amsthm,amsfonts}
\usepackage{tikz-cd}
\usepackage{enumerate}
\usepackage{setspace}
\usepackage{empheq}
\usepackage[all]{xy}
\usepackage{verbatim}
\usepackage[normalem]{ulem}
\usepackage{todonotes}
\usepackage[bookmarks,colorlinks,breaklinks]{hyperref}  
\hypersetup{linkcolor=blue,citecolor=blue,filecolor=blue,urlcolor=blue} 

\numberwithin{equation}{section}
\numberwithin{figure}{section}

\newtheorem{theorem}{Theorem}[section]
\newtheorem{proposition}[theorem]{Proposition}

\newtheorem{lemma}[theorem]{Lemma}

\theoremstyle{definition}

\definecolor{myblue}{rgb}{0.6, 0.9, 1}

\makeatletter

\newcommand{\Rmnum}[1]{\expandafter\@slowromancap\romannumeral #1@}
\makeatother

\definecolor{myblue}{rgb}{0.6, 0.9, 1}
\definecolor{mygreen}{rgb}{0,0,1}
\definecolor{purple}{rgb}{0.6,0.2,1}
\definecolor{orange}{rgb}{0.8,0,0.2}

\newcommand{\bC}{\mathbb{C}}
\newcommand{\bP}{\mathbb{P}}

\newcommand{\Kbar}{\overline{K}}

\newcommand\gf[2]{\genfrac{}{}{0pt}{}{#1}{#2}}

\author{Niki Myrto Mavraki, Harry Schmidt and Robert Wilms}
\email{mavraki@math.harvard.edu}
	\email{harry.schmidt@unibas.ch}
\email{robert.wilms@unibas.ch}
	
\title{Height coincidences in products of the projective line}

\begin{document}
\begin{abstract}
 We consider hypersurfaces in $\bP_1^n$ that contain a generic sequence of small dynamical height with respect to a split map and project onto  $n-1$ coordinates. We show that these hypersurfaces satisfy strong coincidence relations between their points with zero height coordinates. More precisely, it holds that in a Zariski-open dense subset of such a hypersurface $n-1$ coordinates have height zero if and only if all coordinates have height zero. This is a key step in the resolution of the dynamical Bogomolov conjecture for split maps. 
\end{abstract}
\maketitle

\section{Introduction} 

In this note we prove a coincidence relation of heights for hypersurfaces in products of the projective line, which contain many points of small canonical height. 
Let $K$ be a number field or function field of a smooth projective curve over an algebraically closed field. For $i=1,\ldots,n$ let $f_i:\bP_1\to \bP_1$ be degree $d_i\ge 2$ morphisms of the projective line defined over $K$. Set $\Phi=(f_1,\ldots,f_n):\bP_1^n\to \bP_1^n$ and write $$\hat{h}_{\Phi}(z_1,\ldots,z_n)=\sum_{i=1}^n\hat{h}_{f_i}(z_i),$$
where $\hat{h}_{f_i}$ denotes the canonical height on $\mathbb{P}_1$ associated to $f_i$ defined by Call and Silverman \cite{callsilverman}. Our main result shows that the existence of a generic small sequence of points in a hypersurface  with respect to the height $\hat{h}_{\Phi}$ implies coincidence relations between the coordinates of points with height zero. Note that a sequence is generic in a variety $X$ if none of its subsequences is contained in a proper algebraic subvariety of $X$.


\begin{theorem}\label{main} Let $H \subseteq \mathbb{P}_1^{n}$ be an irreducible hypersurface and let $j\in \{1,\ldots,n\}$ be such that $p_j(H)=\bP_1^{n-1}$ for the projection $p_j\colon \mathbb{P}_1^n\to\mathbb{P}_1^{n-1}$ forgetting the $j$-th factor. Suppose that there exists a generic sequence $\{x_{\ell}\}\subseteq H(\overline{K})$ with $\hat{h}_{\Phi}(x_{\ell}) \to 0$ as $\ell\to\infty$. 
Then there exists a closed subset $E_j \subsetneq H$ such that 
\begin{align*}
\{x\in H\setminus E_j(\Kbar)|\hat{h}_{f_i}(\pi_i(x))=0 \text{ for all } i\in \{1,\dots,n\}\setminus\{j\}\}=\{x\in H\setminus E_j(\Kbar)|\hat{h}_{\Phi}(x)=0\},
\end{align*} 
where $\pi_i: H\to \bP_1$ denotes the projection to the $i$-th coordinate.
More precisely,
\begin{align*}  
E_j(\overline{K}) = \{(x_1, \dots, x_n)\in H(\overline{K})~|~ (x_1,  \dots, x_{j-1})\times \mathbb{P}_1(\overline{K})\times (x_{j+1}, \dots, x_n)\subseteq H(\overline{K}) \}.
\end{align*} 
In particular, if $x \in H(\overline{K})$ satisfies $\hat{h}_{f_i}(\pi_i(x)) = 0$ for all $i \neq j$, then there exists $y \in H(\overline{K})$ such that $\hat{h}_\Phi(y) = 0$ and $p_j(y) = p_j(x)$. 
 	\end{theorem}    
 
Our strategy is to interpret points as the intersection of divisors obtained by fixing one coordinate and to use the induction formula by Chambert-Loir and Thuillier \cite{CLTexposition} in the setting of adelic line bundles to compute their heights inductively. To facilitate the induction step we show that the integrals in the induction formula vanish. For this purpose, we apply the equidistribution result by Yuan \cite{BiglinebundlesYuan} and Gubler \cite{Gubler:equidistribution}. We also rely on the fundamental inequality by Zhang \cite{smallpointszhang} and Gubler \cite{bogomolovgubler}, which allows us to translate between the existence of a generic sequence of small points and the vanishing of the top self-intersection of the polarizing adelic line bundle.

For the result it is necessary to take out the exceptional set $E_j$ and the explicit choice of $E_j$ in the theorem is optimal since in this set the $j$-th coordinate is independet of the other coordinates, such that the assertion cannot hold for points in $E_j$. One cannot expect $E_j(\overline{K})$ to be empty in general. For example if $H\subseteq\mathbb{P}_1^3$ is given by the equation $x_1y_1z_1=x_0y_0z_0$ and every $f_i$ is the square map, then $\{0\}\times \{\infty\}\times \mathbb{P}_1(\overline{K})\subseteq E_3(\overline{K})$. 

For $n=2$ the theorem has already been proved by Mimar \cite{mimar} using a different strategy. He relies on the arithmetic Hodge index theorem by Faltings and Hriljac. For $n>2$ Ghioca, Nguyen and Ye \cite[Proposition 5.2]{dynamicalmm} claimed a proof of Theorem \ref{main} for a number field $K$ using the more involved arithmetic Hodge index theorem for adelic line bundles by Yuan and Zhang \cite{arithmetichodgeindexyuanzhang}. Their proof relies on the existence of a non-trivial linear combination of certain line bundles, which is numerically trivial; see \cite[Claim 5.4]{dynamicalmm}. But the latter appears to be in contradiction with the Lefschetz hyperplane theorem; see for example \cite[Theorem 5.21]{arithmetichodgeindexyuanzhang}. The statement of \cite[Proposition 5.2]{dynamicalmm} also differs from our Theorem \ref{main} and would imply that the set $E_j$ is empty if all $p_j$ are surjective. However, this is, for example, in disagreement  with the hypersurface $H$ and the exceptional set $E_3 \subseteq H$  mentioned in the previous paragraph. 

In fact the existence of a generic small sequence in $H$ should impose much stronger rigidity conditions on $H$ than the coincidence relation in Theorem \ref{main} as predicted by the dynamical Bogolomov and Manin--Mumford conjectures of Zhang \cite[Conjecture 4.1.7]{Zhang:distributions}, Ghioca--Tucker \cite[Conjecture 1.2]{GhiocaTucker:reformulation} and Gauthier--Vigny \cite[Conjecture 1]{GVgeometric}. Indeed, if $K$ is a number field, combining Theorem \ref{main} with the complex analytic result of Ghioca--Nguyen--Ye \cite[Theorem 6.1]{dynamicalmm} yields that the dynamical Bogomolov conjecture holds for split maps on $\bP_1^n$ as explained in \cite[Theorems 1.2, 1.3]{dynamicalmm}. If $K$ is a function field with characteristic zero and the maps $f_i$ are not isotrivial, one can use Theorem \ref{main} to infer that the geometric dynamical Bogomolov holds in this setting, as explained in work of the first and second author \cite[Theorem 4.1]{mavraki2022dynamical}.

\section{Preliminaries}
\label{section-preliminaries}
Let us first clarify some notations and recall some results from (arithmetic) intersection theory. Let $K$ be a function field of a smooth projective curve $B$ or a number field, in which case we set $B=\mathrm{Spec}(\mathcal{O}_K)$. We write $M(K)=M(K)_0\cup M(K)_\infty$, where $M(K)_0$ denotes the set of non-archimedean places of $K$ and $M(K)_\infty$ is the set of complex embeddings $K\to \mathbb{C}$ if $K$ is a number field and $M(K)_\infty=\emptyset$ if K is a function field. For every $v\in M_K$ we write $K_v$ for the completion of $K$ with respect to $v$ and we fix an algebraic closure $\overline{K}_v$ of $K_v$. If $v\in M(K)_0$, we write $\mathcal{O}_{\overline{K}_v}$ for the integers of $\overline{K}_v$. Further, we fix a collection of absolute values $|\cdot|_v$ on $K_v$ satisfying the product formula $\prod_{v\in M(K)} |x|_v=1$ for $x\in K\setminus\{0\}$.

Let $X$ be a projective variety over $K$ and $L$ a line bundle on $X$. We denote $X_v$ and $L_v$ for their pullbacks induced by the canonical map $K\to \overline{K}_v$. A metric on $L_v$ is a collection of $\overline{K}_v$-norms on each fiber $L_v(x)$ for $x\in X_v(\overline{K}_v)$. If $v\in M(K)_0$, some of the most important examples of metrics are induced by models. For any $e\in \mathbb{Z}_{>0}$ and any projective flat model $(\widetilde{X_v},\widetilde{L_v})$ of $(X_v,L_v^{\otimes e})$ over $\mathrm{Spec}(\mathcal{O}_{\overline{K}_v})$ we may define the metric $\|\cdot\|_{\widetilde{L_v}}$ on $L_v$ by setting
$$\|\ell\|_{\widetilde{L_v}}=\inf_{a\in \overline{K}_v}\left\{|a|_v^{1/e}~|~\ell\in a \widetilde{L_v}(\widetilde{x})\right\}$$
for all $x\in X_v(\overline{K}_v)$ and all $\ell\in L_v(x)$. Here $\widetilde{x}\in \widetilde{X_v}(\mathcal{O}_{\overline{K_v}})$ denotes the unique extension of $x$. A metric $\|\cdot \|_v$ on $L_v$ is called continuous and bounded if $\log\frac{\|\cdot\|_v}{\|\cdot\|_{\widetilde{L_v}}}$ is continuous and bounded for some model $(\widetilde{X_v},\widetilde{L_v})$.
 
We call a collection $\|\cdot\|=\{\|\cdot\|_v~|~v\in M(K)\}$ of bounded, continuous and $\mathrm{Gal}(\overline{K}_v/K_v)$-invariant metrics an adelic metric, if there is a non-empty open subset $U\subseteq B$ and a model $(\widetilde{X},\widetilde{L})$ of $(X,L)$ over $U$, such that $\|\cdot \|_v$ is the model metric induced by the pullback of $(\widetilde{X},\widetilde{L})$ along $\mathrm{Spec}(\mathcal{O}_{\overline{K}_v})\to U$ for every $v\in U$. The pair $\overline{L}=(L,\|\cdot\|)$ is called an adelic metrized line bundle.
We are mainly interested in adelic metrics obtained by limits of model metrics. Therefore, we call $\overline{L}=(L,\|\cdot\|)$ semipositive, if there is a sequence of models $(\widetilde{X}_n,\widetilde{L}_n)$ of $(X,L^{\otimes e_n})$ over $B$ with $\widetilde{L}_n^{\otimes k_n}$ generated by global sections for some $k_n>0$, such that for all $v\in M(K)_{0}$ the sequence $\log\frac{\|\cdot\|_{\widetilde{\mathcal{M}}_{n.v}}}{\|\cdot\|_v}$ converges to $0$ uniformly in $X(\overline{K}_v)$ and for each $v\in M(K)_{\infty}$ its curvature form $c_1(L,\|\cdot\|_v)=\frac{\partial \bar{\partial}}{\pi i}\log\|s\|_v+\delta_{\mathrm{div}(s)}$ is semipositive in the sense of currents. Here $s$ is any non-zero meromorphic section of $L(\bC)$. We call $\overline{L}$ nef if it is semipositive and the models $(\widetilde{X}_n,\widetilde{L}_n)$ as above can be chosen and equipped with continuous hermitian metrics $\|\cdot\|_{n,v}$ with semipositive curvature form at each $v\in M(K)_\infty$, such that $\|\cdot\|_{n,v}^{1/e_n}$ uniformly converges to $\|\cdot\|_v$ for all $v\in M(K)_\infty$ and $\widetilde{L}_n$ has non-negative (arithmetic) degree restricted to any $1$-dimensional integral subscheme of $\widetilde{X}_n$. 

Next, we describe the semipositive metrics induced by polarized dynamical systems. Let $f\colon X\to X$ be a morphism and $L$ a polarization of $f$, that is an ample line bundle on $X$ such that $f^*L\cong L^{\otimes d}$ for some  $d>1$. We choose a model $(\widetilde{X},\widetilde{L})$ of $(X,L)$ over $B$. The normalization of the composition $X\xrightarrow{f^{\circ m}}X\to \widetilde{X}$ defines a model $\widetilde{X}_m$ of $X$ over $B$ together with a map $f_m\colon \widetilde{X}_m\to \widetilde{X}$ and we set $\widetilde{L}_m=f_m^*\widetilde{L}$. Then $(\widetilde{X}_m,\widetilde{L}_m)$ is a model of $(X,L^{d^m})$ and they define a sequence of model metrics $\|\cdot\|_m$ on $L$ converging to an semipositive adelic metric $\|\cdot\|$. For $v\in M(K)_\infty$ we choose $\|\cdot\|_{0,v}$ on $\widetilde{L}_0$ to be any smooth hermitian metric with semipositive curvature form and we set $\|\cdot\|_{m+1,v}=\frac{1}{d}f^{*}\|\cdot\|_{m,v}$. The adelic metric $\|\cdot\|$ only depends on the isomorphism $f^*L\cong L^{\otimes d}$. The induced adelic metrized line bundle $\overline{L}=(L,\|\cdot\|)$ is nef, see for example \cite[Theorem 6.1.1]{yuan2021adelic}.

Now, let $Z\subseteq X$ be a subvariety of dimension $d=\dim Z$ and $\overline{L}_0,\dots,\overline{L}_d$ semipositive line bundles on $X$. Further, let $s_i$ be a rational section of $L_i$ for every $i$, such that their divisors have no common intersection in $Z$. For any $v\in M(K)$ we define their local intersection number inductively by 
\begin{align}\label{equ_CLT}
\log\|\langle s_0,\dots,s_d|Z\rangle\|_v=\log\|\langle s_0,\dots,s_{d-1}|Z\cap \mathrm{div}(s_d)\rangle\|_v+\int_{Z_{\overline{K}_v}^{\mathrm{an}}}\log\|s_d\|_v \prod_{i=0}^{d-1}c_1(\overline{L}_{i,v}),
\end{align}
where $\langle s_0,\dots,s_d|Z\rangle$ can also be considered as the induced section of the Deligne pairing of $\overline{L}_0,\dots,\overline{L}_d$. We have to make sense of the integral in the case $v\in M(K)_0$.
If $\overline{L}_0,\dots,\overline{L}_n$ are given by models $(\widetilde{X},\widetilde{L}_j)$ of $(X,L_j^{\otimes e})$ over $B$, we can write $\widetilde{s}_d$ for the section of $\widetilde{L}_d$ extending the section $s_d^{\otimes e}$ of $L_d^{\otimes e}$. Then $V=\mathrm{div}(\widetilde{s}_d)-e\cdot\overline{\mathrm{div}(s_d)}^{\mathrm{Zar}}$ is a Weil divisor $V=\sum_{v\in M(K)_0} V_v$ supported in the closed fibres of $\widetilde{X}$ and the integral is defined by the classical intersection number
\begin{align*}
	\int_{Z_{\overline{K}_v}^{\mathrm{an}}}\log\|s_d\|_v \prod_{i=0}^{d-1}c_1(\overline{L}_{i,v})=\frac{1}{e^d}c_1(\widetilde{L}_0)\dots c_1(\widetilde{L}_{n-1})[V_v]\log N(v).
\end{align*}
For the general case one has to take limits.
Although, we defined the integral just formally, it has been proven by Chambert-Loir \cite{berkovichchambertloir} that there is indeed a measure $\prod_{i=0}^{d-1}c_1(\overline{L}_{i,v})$ on the Berkovich space $X_{\overline{K}_v}^{\mathrm{an}}$ and by Chambert-Loir and Thuillier \cite{CLTexposition} that formula (\ref{equ_CLT}) holds with these measures. By the product formula the number
$$(\overline{L}_0\cdots\overline{L}_d|Z)=-\sum_{v\in M(K)}\log \|\langle s_0,\dots,s_d\rangle \|_v$$
does not depend on the choice of the $s_i$'s and it coincides with the global adelic intersection number of $\overline{L}_0,\dots,\overline{L}_d$ as constructed by Zhang \cite{smallpointszhang}, see also \cite{yuan2021adelic}.
Equation (\ref{equ_CLT}) globalizes to the following formula
\begin{align}\label{equ_CLT-global}
	(\overline{L}_0\cdots\overline{L}_d|Z)=(\overline{L}_0\cdots\overline{L}_{d-1}|Z\cap \mathrm{div}(s_d))-\sum_{v\in M(K)}\int_{Z_{\overline{K}_v}^{\mathrm{an}}} \log\|s_d\|_v \prod_{i=0}^{d-1}c_1(\overline{L}_{i,v})
\end{align}
for any non-zero rational section $s_d$ of $L_d$. We write $\overline{L}_0\cdots\overline{L}_d=(\overline{L}_0\cdots\overline{L}_d|X)$. Recall from \cite[Proposition 4.1.1]{yuan2021adelic} that
$\overline{L}_0\cdots\overline{L}_d\ge 0$ if $\overline{L}_0,\dots,\overline{L}_d$ are nef.

Adelic intersection numbers can be used to define heights of subvarieties. For any ample line bundle $\overline{L}$ equipped with a semipositive adelic metric and any subvariety $Z\subseteq X$ we set
$$h_{\overline{L}}(Z)=\frac{(\overline{L}^{\dim Z+1}|Z)}{(\dim Z+1)\deg_L(Z)},$$
where $(\overline{L}^{\dim Z+1}|Z)$ denotes the $(\dim Z+1)$-th adelic self-intersection of $\overline{L}$ on $Z$.
For any $1\le i\le \dim Z$ we define the $i$-th essential minimum by
$$e_i(Z,\overline{L})=\sup_{\gf{Y\subseteq Z}{\mathrm{codim}(Y)=i}}\inf_{x\in (Z\setminus Y)(\overline{K})} h_{\overline{L}}(x).$$
It has been proven by Zhang \cite[Theorem 1.10]{smallpointszhang} in the number field case and by Gubler \cite[Section 4]{bogomolovgubler} in the function field case that we have the fundamental inequality
\begin{align}\label{equ_fundamental}
e_1(Z,\overline{L})\ge h_{\overline{L}}(Z)\ge \frac{\sum_{i=1}^{\dim Z+1}e_i(Z,\overline{L})}{\dim Z+1}.
\end{align}

\section{Proof of Theorem \ref{main}}
Let us first prove several lemmas.
For every $m$ and $I=\{i_1,\dots,i_{|I|}\}\subseteq\{1,\dots,m\}$ we write $\pi_{m,I}\colon \mathbb{P}_1^{m}\to \mathbb{P}_1^{|I|}$ for the projection to the $i_1$-th, $\dots$, $i_{|I|}$-th factors. Especially, we write $\pi_{I}:=\pi_{n,I}$. Further, let $p_{m,i}\colon \mathbb{P}_1^{m}\to\mathbb{P}_1^{m-1}$ be the projection forgetting the $i$-th factor for some $i\le m$. Note that $p_{i}=p_{n,i}$.
Our first lemma will allow us to reduce the proof to the case that $p_j(H)=\bP_1^{n-1}$ for all $j\in\{1,\ldots,n\}$. 
 
  \begin{lemma}\label{reduction}
  	For every irreducible hypersurface $H\subseteq \mathbb{P}_1^n$ there exists a subset $$I=\{i_1,\dots,i_k\}\subseteq \{1,\dots,n\}$$ with $|I|=k$, such that $H$ is of the form $H=\pi_{I}^{-1}(H')$, where $H'\subseteq \mathbb{P}_1^{k}$ is an irreducible hypersurface satisfying $p_j(H')=\mathbb{P}_1^{k-1}$ for every $1\le j\le k$.
  \end{lemma}
  \begin{proof}
  	We prove this by induction on $n$. If $p_j(H)=\mathbb{P}_1^{n-1}$ for all $1\le j\le n$, there is nothing to prove. Thus, let $j$ satisfy $p_j(H)\neq \mathbb{P}_1^{n-1}$. By symmetry we may assume $j=n$. Then $p_n(H)\subseteq \mathbb{P}_1^{n-1}$ is an irreducible hypersurface and by induction there is a subset $I\subseteq \{1,\dots, n-1\}$ with $|I|=k$ for some $1\le k\le n-1$ and an irreducible hypersurface $H'\subseteq \mathbb{P}_1^{k}$ such that $p_n(H)=\pi_{n-1,I}^{-1}(H')$. Taking the preimage with respect to $p_n$, we get $H\subseteq \pi_{I}^{-1}(H')$. But $\pi^{-1}_{I}(H')$ is of codimension $1$ and it is irreducible as it is of the form $\pi^{-1}_{I}(H')\cong \mathbb{P}_1^{n-k}\times H'$. Thus, it must hold $H= \pi_{I}^{-1}(H')$.
  \end{proof}
  By Lemma \ref{reduction} we can reduce the proof of Theorem \ref{main} to the case where $p_j(H)=\mathbb{P}_1^{n-1}$ for all $1\le j\le n$. Indeed, if $H=\pi_{I}^{-1}(H')$ with $I\neq \{1,\dots, n\}$, we may assume by symmetry, that $I=\{1,\dots, k\}$ such that $H=H'\times \mathbb{P}_1^{n-k}$. If the theorem holds for $H'$, then it also holds for $H$.

 To fix notation, we let $\overline{L}_i$ be the line bundle $\mathcal{O}(1)$ on $\bP_1$ equipped with canonical adelic metric associated with $f_i$ as described in Section \ref{section-preliminaries}. 
 Using additive notation, we let 
  $$\overline{L} = \pi_1^*\overline{L}_1 + \cdots + \pi_n^*\overline{L}_n,$$
  which is an adelic line bundle on $H$. On points the height $\hat{h}_{f_i}$ coincides with $h_{\overline{L}_i}$ and the height $\hat{h}_{\Phi}$ coincides with $h_{\overline{L}}$. In particular, the height $h_{\overline{L}}$ takes non-negative values at points, such that $e_i(H,\overline{L})\ge 0$ for all $i=1,\ldots,n$
  
   We  record the following consequence of the fundamental inequality (\ref{equ_fundamental}). 
  \begin{lemma}\label{zhang}
  	Under the assumptions of Theorem \ref{main} we have 
  	\begin{align*}
  		\overline{L}^n= \pi_{i_1}^*\overline{L}_{i_1}\cdots \pi_{i_n}^*\overline{L}_{i_n}=0,
  	\end{align*}
  	for every tuple $(i_1,\dots,i_{n})\in \{1,\dots,n\}^n$. 
  \end{lemma}     
  
  \begin{proof}
  	Our assumption on the existence of a generic small sequence yields that $e_i(H,\overline{L})=0$ for all $i=1,\ldots,n$. 
  	Since $\overline{L}$ is an ample line bundle equipped with a semipositive adelic metric, the fundamental inequality (\ref{equ_fundamental}) yields 
  	$$0=(\dim H+1)\deg_L(H) h_{\overline{L}}(H)=\overline{L}^n=(\pi_1^*\overline{L}_1 + \cdots + \pi_n^*\overline{L}_n)^n.$$
  	As the $\overline{L}_i$'s are nef, their pullbacks by $\pi_i$ are also nef. Thus, the power on the right hand side decomposes as a sum of non-negative terms. Hence all terms vanish. 
  \end{proof}     
  
  We apply arithmetic equidistribution to infer the following. 
\begin{lemma}\label{measures} Let $H$ be as in Theorem \ref{main} and assume further that $p_j(H) = \mathbb{P}_1^{n-1}$ for all $j \in \{1, \dots, n\}$. There exist $c_{jk} \in \mathbb{Q}^*$ for $j,k \in \{1, \dots, n\} $  such that
	$$\prod_{i = 1, i \neq j}^nc_1(\pi_i^*\overline{L}_{i,v}) = c_{jk}\prod_{i = 1, i \neq k}^nc_1(\pi_i^*\overline{L}_{i,v})$$ 
	for any $v\in M(K)$.
\end{lemma} 
\begin{proof} For every rational number $\epsilon$ and every $j\in \{1,\dots,n\}$ we define the adelic $\mathbb{Q}$-line bundle
     $$\overline{M}_{j,\epsilon}=\epsilon \pi_j^* \overline{L}_i+\sum_{i=1,i\neq j}^n\pi_i^*\overline{L}_i.$$
     Since $\sum_{i=1}^n\pi_i^*L_i$ is ample as the restriction of an ample line bundle on $\mathbb{P}_1^n$ to $H$ and every $\pi_i^*L_i$ is nef as a pull back of an ample line bundle, we get that the underlying line bundle $M_{j,\epsilon}$ of each $\overline{M}_{j,\epsilon}$ is ample if $\epsilon>0$. As a positive linear combination of semipositive metrics, the metric on $\overline{M}_{j,\epsilon}$ is also semipositive for $\epsilon>0$. 
     The small sequence $\{x_\ell\}\subseteq H(\overline{K})$ with respect to $\overline{L}$ is also a small sequence with respect to $\overline{M}_{j,\epsilon}$ for all $j\in \{1,\dots,n\}$ and rationals $\epsilon>0$. By Lemma \ref{zhang} $\overline{M}_{j,\epsilon}^n=0$. Thus it follows from  Yuan's equidistribution theorem \cite[Theorem~3.1]{BiglinebundlesYuan} (for the case of number fields) or Gubler's equidistribution theorem \cite[Theorem~6.3]{Gubler:equidistribution} (for the case of function fields), that 
     $$\frac{c_1(\overline{M}_{j,\epsilon,v})^{n-1}}{\mathrm{deg}_{M_{j,\epsilon}}(H)}=\frac{c_1(\overline{L}_v)^{n-1}}{\mathrm{deg}_{L}(H)}=\frac{c_1(\overline{M}_{k,\epsilon,v})^{n-1}}{\mathrm{deg}_{M_{k,\epsilon}}(H)}$$
     for all $j,k\in \{1,\dots,n\}$ and all rationals $\epsilon>0$. By the assumption that $\pi_j(H) =\mathbb{P}_1^{n-1}$ for all $j\in\{1,\dots,n\}$, the line bundle $M_{j,0}$ is big as a pull back via a birational morphism of an ample line bundle. This means that $\deg_{M_{j,0}}(H)>0$. Note that $c_1(\overline{M}_{j,\epsilon,v})^{n-1}$ and $\deg_{M_{j,\epsilon}}(H)$ are continuous in $\epsilon$, as they are induced by multi-linear forms. Hence, we get
     \begin{align*}
     c_1(\overline{M}_{j,0,v})^{n-1}&=\lim_{\epsilon\to 0} c_1(\overline{M}_{j,\epsilon,v})^{n-1}=\lim_{\epsilon\to 0}\frac{\deg_{M_{j,\epsilon}}(H) c_1(\overline{M}_{k,\epsilon,v})^{n-1}}{\deg_{M_{k,\epsilon}}(H)}=\frac{\deg_{M_{j,0}}(H)}{\deg_{M_{k,0}}(H)}c_1(\overline{M}_{k,0,v})^{n-1}
     \end{align*}
     for all $j,k\in\{1,\dots,n\}$.
     Setting $c_{jk}=\frac{\deg_{M_{j,0}}(H)}{\deg_{M_{k,0}}(H)}$ the lemma follows, as we have
     $$\prod_{i=1,i\neq j}^n c_1(\pi_i^*\overline{L}_{i,v})=\frac{c_1(\overline{M}_{j,0,v})^{n-1}}{(n-1)!} =c_{jk} \frac{c_1(\overline{M}_{k,0,v})^{n-1}}{(n-1)!}=c_{jk}\prod_{i=1,i\neq k}^nc_1(\pi_i^*\overline{L}_{i,v}),$$
     since $c_1(\pi_i^*\overline{L}_{i,v})^2 \equiv 0$ for all $i \in\{1, \dots, n\}$. 
\end{proof}

In what follows let $a \in \mathbb{P}_1(\overline{K})$. After replacing $K$ by a finite extension we may assume that $a \in \mathbb{P}_1(K)$. We let $1_a$ be the canonical section of $\mathcal{O}_{\mathbb{P}_1}(a)$ vanishing  on $a$. 
The following lemma is the first consequence of formula (\ref{equ_CLT-global}) we plan to use. 
\begin{lemma} \label{integrals}
Under the assumptions of Theorem \ref{main}, assume further that $p_j(H) = \mathbb{P}_1^{n-1}$ for all $j \in \{1, \dots, n\}$. 
Let $i\in\{1,\ldots,n\}$ and $k\in\{1,\ldots,n\}\setminus\{i\}$. If  $a\in\bP_1(K)$ is such that 
$\hat{h}_{f_i}(a)=0$, 
then 
	$$ \sum_{v \in M(K)}\int_{H^{\mathrm{an}}_{\overline{K}_v}}\log ||\pi_i^*1_a||_{\pi_i^{*}\overline{L}_{i,v}}\prod_{j = 1, j \neq k}^{n}c_1(\pi_j^*\overline{L}_{j,v}) = 0.$$
\end{lemma}
\begin{proof} 
For symmetry reasons we may assume that $i = 1, k = n$. 
By Lemma \ref{zhang} and Equation (\ref{equ_CLT-global}) we get 
\begin{align*}
0&=\pi_1^*\overline{L}_1\cdots \pi_{n-1}^*\overline{L}_{n-1}\cdot\pi_1^*\overline{L}_1 \\
&=  (\pi_1^*\overline{L}_1\cdots \pi_{n-1}^*\overline{L}_{n-1}|\mathrm{div}(\pi_1^*1_a))
- \sum_{v \in M(K)}\int_{H^{\mathrm{an}}_{\overline{K}_v}}\log ||\pi_1^*1_a||_{\pi_1^{*}\overline{L}_{1,v}}\prod_{j = 1}^{n-1}c_1(\pi_j^*\overline{L}_{j,v}).
\end{align*}
     By the fundamental inequality \eqref{equ_fundamental} and arguing as in Lemma \ref{zhang} it also holds that 
	$$ (\pi_1^*\overline{L}_1\cdots \pi_{n-1}^*\overline{L}_{n-1}|\mathrm{div}(\pi_1^*1_a)) = 0$$
	because $\mathrm{div}(\pi_1^*1_a)=H \cap  \{x_1 = a\}$ contains a generic small sequence with respect to $\pi_1^*\overline{L}_1 + \cdots + \pi_{n-1}^*\overline{L}_{n-1}$ as $\hat{h}_{f_1}(a)=0$ and the projection of $H$ to its first $n-1$ coordinates is dominant. 
    The lemma follows.
\end{proof}
  \begin{proposition}\label{inductionstep} 
  	Let $H$ be as in Theorem \ref{main} and assume further that $p_j(H) = \mathbb{P}_1^{n-1}$ for all $j \in \{1, \dots, n\}$. 
    Let $i\in\{1,\ldots,n\}$. 
    If $a \in \mathbb{P}_1(K)$ satisfies $\hat{h}_{f_i}(a)=0$, then 
  	$(\overline{L}^{n-1}| \mathrm{div}(\pi_i^*1_a))=0 .$
  \end{proposition}
\begin{proof} 
Let  $i\in\{1,\ldots,n\}$ and $a\in \mathbb{P}_1(K)$ with $\hat{h}_{f_i}(a)=0$.
By equation (\ref{equ_CLT-global}) we infer 
\begin{align*}
  \overline{L}^{n-1}\cdot\pi_i^{*}\overline{L}_i =  (\overline{L}^{n-1}|\mathrm{div}(\pi_i^*1_a))  -(n-1)!\sum_{k=1}^n\sum_{v \in M(K)}\int_{H^{\mathrm{an}}_{\overline{K}_v}}\log \|\pi_i^*1_a\|_{\pi_i^{*}\overline{L}_i ,v}\prod_{j \neq k }c_1(\pi_j^*\overline{L}_{j,v}). 
  \end{align*}
The proposition now follows combining Lemmas \ref{zhang}, \ref{measures} and \ref{integrals}. 
\end{proof}

We are now ready to prove Theorem \ref{main}.

\begin{proof}[Proof of Theorem \ref{main}:] 
Let $H$ be as in the statement of the theorem.
By the remark after Lemma \ref{reduction} we may assume that $p_i(H) = \mathbb{P}_1^{n-1}$ for all $i\in \{1, \dots,n\}$. We will argue by induction on $n$. For $n=1$ the statement is trivial. Let $n\ge 2$.
Permuting the coordinates, we may assume $j=1$. Recall that
	$$E_1(\overline{K})=\{(x_1,\dots,x_n)\in H(\Kbar)~|~ \mathbb{P}_1(\overline{K})\times (x_2,\dots,x_{n})\subseteq H(\overline{K})\}.$$
	We fix an $a=(a_1,\dots,a_n)\in H(\overline{K})\setminus E_1(\overline{K})$,
    such that $\hat{h}_{f_2}(a_2)= \cdots= \hat{h}_{f_{n}}(a_{n})=0$. If this does not exist, there is nothing to prove. We have to show that also $\hat{h}_{f_1}(a_1)=0$. From Proposition \ref{inductionstep} it follows that $(\overline{L}^{n-1}|_{\text{div}(\pi_n^* 1_{a_n})}) = 0$. By the fundamental inequality (\ref{equ_fundamental}) we infer that $H \cap \{x_n = a_n\}$ contains a generic small sequence with respect to $\pi_1^{*}\overline{L}_1+\cdots+\pi_{n-1}^{*}\overline{L}_{n-1}$. As $H\cap \{x_n=a_n\}$ lies in the image of the morphism
	$$\varphi\colon \mathbb{P}_1^{n-1}\to\mathbb{P}_1^n,\qquad (x_1,\dots,x_{n-1})\mapsto (x_1,\dots,x_{n-1},a_n),$$
	we may consider it as a hypersurface $H_1=\varphi^{-1}(H\cap \{x_n=a_n\})\subseteq \mathbb{P}_1^{n-1}$. Let $H_1=\bigcup_{i=1}^r H_{1,i}$ be the decomposition of $H_1$ into irreducible components. Let $H_{1,j}$ be a component such that $(a_1,\dots,a_{n-1})\in H_{1,j}$. It also holds that $\mathbb{P}_1(\overline{K})\times (a_2, \dots, a_{n-1}) $ is not contained in $H_{1,j}(\overline{K})$ since otherwise $(a_1, \dots, a_{n}) \in E_1(\overline{K})$. In order to check that $H_{1,j}$ satisfies the assumptions of Theorem \ref{main} we note that if we had $p_{n-1,1}(H_{1,j})\neq \mathbb{P}_1^{n-2}$, then we would have $H_{1,j}=\mathbb{P}_1\times H'$ for some irreducible hypersurface $H'\subseteq \mathbb{P}_1^{n-2}$. But this contradicts that $\mathbb{P}_1(\overline{K})\times(a_2, \dots, a_{n-1})$ is not contained in $H_{1,j}(\overline{K})$. Hence, we have $p_{n-1,1}(H_{1,j})=\mathbb{P}_1^{n-2}$, such that we can apply the theorem to $H_{1,j}$ by induction. This yields that $\hat{h}_{f_1}(a_1)=0$ as desired.
\end{proof}

\section*{Acknowledgements}
N.M.M. acknowledges the support from NSF grant DMS-2200981. R.W. acknowledges support from the SNF grant ``Diophantine Equations: Special Points, Integrality, and Beyond'' (n$^\circ$ 200020\_184623).

\bibliography{Galoisbounds}
\bibliographystyle{alpha}

\end{document}